\documentclass[11pt,a4paper,oneside,reqno]{amsart}

\newtheorem{theorem}{Theorem}[section]
\newtheorem{lemma}[theorem]{Lemma}
\usepackage{setspace}
\theoremstyle{definition}
\newtheorem{definition}[theorem]{Definition}

\newtheorem{proposition}[theorem]{Proposition}
\newtheorem{corollary}[theorem]{Corollary}
\newtheorem{notation}[theorem]{Notation}

\newtheorem{remark}[theorem]{Remark}
\newtheorem{Notation and Remark}[theorem]{Notation and Remark}
\usepackage{amsmath}
\long\def\comment#1{\relax}
\long\def\commented_out#1{\relax}
\usepackage{enumitem}
\usepackage{amssymb}
\usepackage{lipsum}
\usepackage{latexsym}
\usepackage{fullpage}
\usepackage{xspace} 
\DeclareMathOperator{\alfa}{\alpha}

\newcommand{\N}[2][k]{N_{#2}(<#1)}
\newcommand{\Nd}[2][k]{N'_{#2}(<#1)}

\newcommand{\Ralfak}[1][R]{{#1}[\alfa_1,\ldots,\alfa_k]}

\newcommand{\Null}[1][R]{N_{#1}}
\newcommand{\Nulld}[1][R]{N'_{#1}}
\DeclareMathOperator{\im}{Im}

\newcommand{\Stabk}[1][R]{St_{\alfa_1,\ldots,\alfa_k}(#1)}
\newcommand{\PrPol}[1][R]{\mathcal{P}(#1)}
\newcommand{\PolFun}[1][R]{\mathcal{F}(#1)}

\DeclareMathOperator{\quv}{\xspace{ }\triangleq\xspace{ }}
\DeclareMathOperator{\nquv}{\xspace{}\not\triangleq\xspace{}}

\usepackage{mathtools}

\usepackage{epsf}
\usepackage{layout}
\usepackage{hyperref}                      
\begin{document}
	\openup .3em
	%\lipsum[3]
	%\layout
	\title[]{Polynomial functions over  dual numbers of several variables}
	
	%    author one information
	%\addtolength{\itemsep}{2pt}
	%\addtolength{\parsep}{2pt}
	%\addtolength{\itemsep}{2pt}
	
	\author{Amr Ali Abdulkader Al-Maktry}
	\address{\hspace{-12pt}Department of Analysis and Number Theory (5010) \\
		Technische Universit\"at Graz \\
		Kopernikusgasse 24/II \\
		8010 Graz, Austria}
	\curraddr{}
	\email{almaktry@math.tugraz.at}
	\curraddr{}
	\email{}
	%\thanks{A.~A.~Al-Maktry's research was supported by the Austrian Science Fund FWF grant P27816-N26}
	
	%    author two information

	%    \subjclass is required.
	
	
	\date{}
	
	\dedicatory{}
	\maketitle
	%    Abstract is required.
	
	\section*{Abstract}
	%Let $k\in \mathbb{N}\setminus\{0\}$. 
	Let $k$ be a positive integer. For a commutative ring $R$,  the ring of dual numbers of  $k$ variables over $R$
	is the quotient ring $R[x_1,\ldots,x_k]/ I $, where $I$ is the ideal generated by the set $\{x_ix_j\mid  i,j\in \{1, \ldots,k\}\}$. This ring can be viewed as $R[\alpha_1,\ldots,\alpha_k]$ with $\alfa_i \alfa_j=0$, where $\alfa_i=x_i+I$ for $1\le i,j\le k$. We investigate the polynomial functions of  $R[\alpha_1,\ldots,\alpha_k]$ whenever $R$ is a finite commutative ring. We derive counting formulas for the number of polynomial functions and polynomial permutations on $R[\alpha_1,\ldots,\alpha_k]$ depending on the order of the pointwise stabilizer of the  subring of constants $R$ in the group of polynomial permutations of $R[\alpha_1,\ldots,\alpha_k]$. Further, we show that the stabilizer group of $R$ is independent of the number of variables $k$. Moreover, we prove that a function $F$ on $R[\alpha_1,\ldots,\alpha_k]$ is a polynomial function if and only if
	a system of linear equations  on $R$ that depends on $F$ has a solution.\\

	\section{Introduction}\label{CHSIntro} 
	
	Let $R$ be a finite commutative ring with unity. Then a function $F\colon R\longrightarrow R$ is said to be a polynomial function on  $R$ if there exists a polynomial $f\in R[x]$ such that
	$f(a)=F(a)$ for every $a\in R$. In this case, we say that $F$ is the induced function of $f$ on $R$ and $f$ represents (induces) $F$. Moreover, if $F$ is a bijection,  we say that $F$ is a \emph{polynomial permutation} and $f$ is a \emph{permutation polynomial}. 
	If $R$ is a finite field, it can be shown easily by using  Lagrange  interpolation that every function on $R$ is a polynomial function. The situation is different when $R$ is not a field and it is somewhat more complicated to study the properties of polynomial functions on such a ring. We denote by $\PolFun[R]$   the set of polynomial functions on $R$,  which is evidently a monoid under the composition of functions. Moreover, its subset of polynomial permutations forms a group  and we denote it by $\PrPol[R]$.
	
	Kempner~\cite{Residue} was the first mathematician who studied polynomial functions on a finite ring which is not a field.
	He studied extensively the polynomial functions on $\mathbb{Z}_m$, the ring of integers modulo $m$. However, his arguments and results  were somewhat lengthy and sophisticated. Therefore, for a long time some researchers \cite{pol1,pol2,pol3} followed up his work, obtained simpler proofs and contributed to the subject as well.
	Meanwhile, some others were interested in the group of  polynomial permutations modulo $p^n$ \cite{per1,per2}. Other mathematicians have generalized the concepts of polynomial functions on $\mathbb{Z}_m$ into other rings, for example, local principal ideal  rings~\cite{Necha} and Galois rings~\cite{gal}. Later, Frisch~\cite{suit} characterized the polynomial functions over a more general class of local rings.  Surprisingly, all rings examined in ~\cite{gal,Necha,Residue} are contained in this class.
	
	In a recent paper \cite{Haki}, the authors considered  the polynomial functions of the ring %of dual numbers modulo $m$.
	$R[x]/(x^2)$, the ring of dual numbers over $R$. In particular, they examined extensively the properties of the polynomial functions
	on dual numbers over the integers modulo $p^n$  by relating them to the polynomial functions modulo $p^n$. However, dual numbers over finite local rings that are not fields  are not contained in the class of rings covered in \cite{suit}, (see Proposition ~\ref{CHSsutt}).
	% except for some trivial cases.

	It should be mentioned that   some mathematicians examined the properties of polynomial functions on weaker structures such as semi groups~\cite{semi} and monoids~\cite{mon}. 
	
	The importance of studying polynomial functions emanates from their intrinsic applications in other areas. For example, permutation polynomials modulo $p^n$ have been employed widely in computer science
	(see for example \cite{Com1,Com2}). Also, they occur  as isomorphisms of combinatorial objects with vertex set  $\mathbb{Z}_{p^n}$ \cite{Cobj1,Cobj2}. For this reason, we think
	that investigating the polynomial functions on new structures will give a good chance for new applications
	to come out. 
	
	In this paper, we are interested in the polynomial functions of the ring of dual numbers of several variables  over  %general finite local rings
	a finite local ring $R$, that is, the ring  $R[x_1,\ldots,x_k]/ I $, where $I$ is the ideal generated by the set $\{x_ix_j\mid  i,j\in\{1,\ldots,k\}\}$, alternatively, the ring $\Ralfak$ with  $\alfa_i \alfa_j=0$.  We relate the properties of the polynomial functions on such a ring to the polynomial functions on  $R$ (see for example Theorems~\ref{CHS4} and \ref{CHSGenper}). Furthermore, we show that the pointwise stablizer of $R$ in the group of polynomial permutations on $\Ralfak$ plays an essential role in the counting formulas of the polynomial functions and the polynomial permutations on $\Ralfak$. More generally, we show that the properties of the polynomial functions on $R[x]/(x^2)$ discussed in \cite{Haki}  can be carried over to those on the ring $\Ralfak$.
	
	Here is a summary of the paper. Section~\ref{CHSsc2} contains  some basics and notations. In Section~\ref{CHSsc3}, we 
	characterize  null polynomials on $\Ralfak$.
	Section~\ref{Sec04} considers  permutation  polynomials and polynomial permutations on $\Ralfak$.
	% and we develop the ideas needed in  section~\ref{CHSsc4}.
	Then, in Section~\ref{CHSsc4}, we  consider a group of polynomial permutations on $\Ralfak$  that stabilizes (fixes) the elements of $R$ pointwise, and derive some counting formulas in terms of the order of this stabilizer group. Finally, we obtain necessary and sufficient conditions  for polynomial functions on $\Ralfak$ in section~\ref{CHSsec5}

	\section{Basics}\label{CHSsc2}
	In this section, we introduce some definitions and facts that  appear in the paper frequently. Throughout this paper,  let $k$ be a positive integer, and for $f\in R[x]$ let $f'$ denote its  first formal derivative.
	\begin{definition}\label{CHSequvfun}
		Let $A$ be a commutative ring, and  $f\in A[x]$. Then:
		\begin{enumerate}
			\item The polynomial $f$ gives rise to a polynomial function on $A$ by substitution for the variable. We use the notation  $[f]_A$ for this function. We just write $[f]$ instead of $[f]_A$, when there is no confusion.
			\item If $[f]$ is a permutation of $A$, then we call  $[f]$ is a polynomial permutation and $f$ a permutation polynomial on $A$.
			\item If $g\in A[x]$ and $[f]=[g]$, this means that $f$ and $g$ induce the same function  on $A$ and we abbreviate this with $f \quv  g$ on $A$.
			\item We define $$\PolFun[A]=\{[f] \mid f\in A[x]\}, \text{ and } $$
			$$\PrPol[A]=\{[f] \mid [f] \text{ is a permutation of }A \text{ and }   f\in A[x]\}. $$
			
			\item If $R$ is a subring of $A$, and $f\in R[x]$, then $f$ gives rise to polynomial functions on $R$ and as well as on $A$. To distinguish between them we write $[f]_R$ and $[f]_A$.
			
		\end{enumerate}
	\end{definition}
	\begin{remark}
		Clearly, $\quv$ on $A$ is an equivalence relation on $A[x]$. Also, there is a   bijective correspondence between the equivalence classes of $\quv$ and the polynomial
		functions on $A$. In particular, if $A$ is   finite, then the number $|\PolFun[A]|$
		of different polynomial functions on $A$  equals the number of 
		equivalence classes of $\quv$ on $A[x]$.
		
	\end{remark}
	
	\begin{definition}\label{CHS001}
		For a commutative ring $R$, the ring of dual numbers of $k$ variables over $R$ is the quotient
		ring $R[x_1,\ldots,x_k]/I$, where $I$
		is the ideal generated by the set $\{x_ix_j\mid  i,j\in\{1,\ldots,k\}\}$. We write 
		$\Ralfak$ for  $R[x_1,\ldots,x_k]/I$, where $\alpha_i$ represents $x_i+I$.
	\end{definition}
	\begin{remark}
		Note that every  element of  $R[x_1,\ldots,x_k]/I$ has a unique representation as an $R$-linear combination of $1,\alfa_1,\ldots,\alfa_k$. That is, $\Ralfak$ is a free $R$-algebra with basis $\{1,\alfa_1,\ldots,\alfa_k\}$. We call the coefficient of $1$ the ``constant coefficient''. Also, $R$ is canonically embedded as a subring in 
		$\Ralfak$ by $r\rightarrow r \cdot 1$, and we have
		\[\Ralfak=\{r_0+\sum\limits_{i=1}^{k}r_i\alfa_i\mid  r_0,r_i\in R, \text{ with }\alfa_i \alfa_j=0 \text { for } 1\le i,j\le k \}.\] 
		%where $\alfa_i$ stands for $x+I $, for $i=1,\ldots,k$.
		It follows from this that  every polynomial $f\in \Ralfak$ has a unique representation  $f =f_0 +\sum\limits_{i=1}^{k}f_i \alfa_i$, where $f_0,f_1,\ldots,f_k\in R[x]$.
	\end{remark}
	
	The following proposition   summarizes some properties  of $\Ralfak$ whose proof is immediate  from Definition~\ref{CHS001}.
	
	\begin{proposition}\label{CHS0}
		Let $R$ be a commutative ring. Then the following hold. 
		
		\begin{enumerate}
			\item
			For $a_0,\ldots,a_k,b_0,\ldots,b_k\in R$, we have:
			
			\begin{enumerate}
				\item 
				$(a_0+\sum\limits_{i=1}^{k}a_i\alfa_i)(b_0+\sum\limits_{i=1}^{k}b_i\alfa_i)=a_0b_0+\sum\limits_{i=1}^{k}(a_0b_i+b_0a_i)\alfa_i$;
				\item
				$a_0+\sum\limits_{i=1}^{k}a_i\alfa_i$ is a unit in $\Ralfak$ if and only if 
				$a_0$ is a unit in $R$. In this case,\\ 
				$(a_0+\sum\limits_{i=1}^{k}a_i\alfa_i)^{-1}=a_0^{-1}-\sum\limits_{i=1}^{k}a_0^{-2}a_i\alfa_i$.
			\end{enumerate}
			\item
			$\Ralfak$ is a local ring if and only if $R$ is a local ring.
			\item 
			If $R$ is a local ring with a maximal ideal $ \mathfrak{m}$ 
			of nilpotency $n$, then $\Ralfak$ is a local ring whose maximal 
			ideal $ \mathfrak{m}+\sum\limits_{i=1}^{k}\alfa_i R$ has nilpotency $n+1$.
		\end{enumerate}
	\end{proposition}
	We  use the following lemma frequently.
	\begin{lemma}\label{CHS02} \label{CHS3} \label{CHS21}
		Let $R$ be a commutative ring and %$\Ralfak$ the ring of dual numbers	of $R$ as in Definition~\ref{CHS001}.
		$a_0,\ldots,a_k\in R$.
		\begin{enumerate}
			\item
			If $f\in R[x]$,
			then
			\[
			f(a_0+\sum\limits_{i=1}^{k}a_i\alfa_i)=f(a_0)+\sum\limits_{i=1}^{k} a_if'(a_0)\alfa_i.
			\]
			\item
			If $f\in {\Ralfak}[x]$ and     $f_0,\ldots, f_k$ are the unique polynomials in $R[x]$
			such that $f =f_0 +\sum\limits_{i=1}^{k}f_i \alfa_i$, then 
			\[			f(a_0+\sum\limits_{i=1}^{k}a_i\alfa_i)=f_0(a_0)+ \sum\limits_{i=1}^{k}(a_if_0'(a_0)+f_i(a_0))\alfa_i.
			\]
		\end{enumerate}
	\end{lemma}
	\begin{proof}
		(1) Follows from Taylor expansion and the fact that $\alpha_i\alpha_j=0$ for $1\le i,j\le k$.\\
		(2) Follows from (1). 
	\end{proof}
	
	The above lemma yields  necessary conditions for a function 
	$F\colon \Ralfak\longrightarrow \Ralfak$ to be a polynomial function.
	\begin{corollary}\label{CHSneccondit}
		Let $F\colon \Ralfak\longrightarrow \Ralfak$ be a polynomial function and  
		let $a_0,a_{1},\ldots,a_k\in R$.
		%	\[F(a_0+\sum\limits_{i=1}^{k}a_i \alfa_i)=c_0+\sum\limits_{i=1}^{k}c_i \alfa_i,\! \quad \text{and} \quad F(b_0+\sum\limits_{i=1}^{k}b_i \alfa_i)=d_0+\sum\limits_{i=1}^{k}d_i \alfa_i.		\]
		Then:
		\begin{enumerate}
			\item The constant coefficient of  $F(a_0+\sum\limits_{i=1}^{k}a_i \alfa_i)$ depends only on $a_0$;
			\item The  coefficient of $\alpha_i$ in $F(a_0+\sum\limits_{i=1}^{k}a_i \alfa_i)$ depends only on $a_0$ and $a_i$.
		\end{enumerate}
		
	\end{corollary}
	\begin{definition}\cite{suit}.\label{CHS1}
		Let $R$ be a finite commutative local ring with a maximal ideal  $\mathfrak{m}$ and $ {L}\in\mathbb{N}$ minimal with $\mathfrak{m}^{L}=(0)$. We call $R$ \textit{suitable}, if for all $a, b\in R$ and all
		$l\in\mathbb{N}$, $ab\in\mathfrak{m}^l\Rightarrow a\in \mathfrak{m}^i$ and $b\in \mathfrak{m}^j$ with $i+j\geq$ min$(L,l)$.
	\end{definition}
	
	The following proposition shows that $\Ralfak$ is not in the class of rings covered in~\cite{suit} unless $R$ is a finite field. 
	
	\begin{proposition}\label{CHSsutt} 
		Let $R$ be a finite local ring. Then $\Ralfak$ is suitable if and only if $R$ is a finite field.
	\end{proposition}
	\begin{proof} Since $R$ is a local ring with a maximal ideal $\mathfrak{m}$ and nilpotency $n$,  $\Ralfak$ is a local ring with maximal ideal $\mathfrak{m}_1=\mathfrak{m}+\sum\limits_{i=1}^{k}\alfa_i R$ and  nilpotency $L=n+1$ by Proposition \ref{CHS0}. Now if $R$ is a field, the result follows easily since $\mathfrak{m}_1^2=(0)$. If $R$ is not
		a field, we notice that $L=n+1>2$ and  $\alfa_1 \in \mathfrak{m}_1\setminus \mathfrak{m}_1^2$, but $\alfa_1^2=0\in\mathfrak{m}_1^{n+1}$. Hence $\Ralfak$ is not suitable, when $R$ is not a field. 
	\end{proof}
	
	\section{Polynomial functions  on $\Ralfak$}\label{CHSsc3}
	From now on, let $R$ be a finite commutative ring with unity.
	In this section, we determine when a given polynomial is  a null polynomial on $\Ralfak$, and whether two polynomials induce the same function on $\Ralfak$. Then we apply these results to obtain a counting formula  for the number of polynomial functions on $\Ralfak$, depending on the indices of the ideals $\Null,\Nulld$   in $R[x]$ (defined below).
	%Later, we dedicate the last part of this section to the group of polynomial permutations on $\Ralfak$, characterize permutation polynomials and provide supplementary results about this group.
	\begin{definition} \label{CHSnulldef}
		\begin{enumerate}
			\item A polynomial $f \in R[x]$ is called a null polynomial on $R$ if $f$ induces the zero function.  	
			\item 	We define $\Null, \Null'$ as:
			\begin{enumerate}
				\item  $\Null=\{ f\in R[x]\mid  f \quv 0 \text{ on } R\}$;
				\item   
				$\Null'=\{ f\in R[x]\mid  f \quv 0\text{ and } 
				f' \quv 0 \text{ on } R\}$. 
			\end{enumerate}
		\end{enumerate}
	\end{definition}
	\begin{remark}\label{nulremark}
		It is evident that $\Null$ and $\Nulld$ are ideals of $R[x]$ with $\Nulld\subseteq \Null$.  Also,  $f \equiv g  \mod \Null$  if and only if $[f]=[g]$; that is polynomial functions on $R$ are in bijective correspondence with residue classes $\mod \Null$. In particular, 
		$|\PolFun|=\left[R[x]\colon \Null\right]$.
	\end{remark}
	\begin{lemma}\label{CHS31} 
		Let $f\in R[x]$. Then:
		\begin{enumerate}
			\item
			$f$ is a null polynomial on $\Ralfak$ if and only if 
			$f\in\Nulld$;
			\item
			$f\alfa_i$ is a null polynomial on $\Ralfak$ for every $1\le i\le k$ if and only if 
			$f\in\Null$.
		\end{enumerate}
	\end{lemma}
	
	\begin{proof} (1)   By Lemma~\ref{CHS02}, for every $a_0,\ldots,a_k \in R$,
		$f(a_0+\sum\limits_{i=1}^{k}a_i \alfa_i)=f(a_0)+ \sum\limits_{i=1}^{k}a_if'(a_0)\alfa_i$. 
		Thus  the fact that $f$ is a null polynomial on $\Ralfak$ is equivalent to 
		$$f(a_0+\sum\limits_{i=1}^{k}a_i \alfa_i)=f(a_0)+ \sum\limits_{i=1}^{k}a_if'(a_0)\alfa_i= 0 \text{ for all }a_0,\ldots,a_k\in R.$$ 
		But
		this is  equivalent to $f(a_0)=0$ and $a_if'(a_0)= 0$ for all $a_0,a_i\in R$ and $i=1,\ldots,k$, 
		which implies that  $f(a_0)=0$ and $f'(a_0)= 0$  for all $a_0\in R$. 
		Hence $f$ and $f'$ are null polynomials on  $R$, which means that $f\in\Nulld$.\\
		(2) Follows immediately from Lemma~\ref{CHS02}.  
	\end{proof}
	\begin{theorem} \label{CHS4} 
		Let $\Null$ and $\Nulld$ as in Definition~\ref{CHSnulldef}, and	let		
		$f =f_0 +\sum\limits_{i=1}^{k}f_i \alfa_i$, where 
		$f_0,\ldots, f_k \in R[x]$. 
		Then $f$ is a null polynomial on $\Ralfak$ if and only if  
		$f_0\in\Nulld$ and $f_i\in\Null$ for $i=1,\ldots,k$.
	\end{theorem}
	
	\begin{proof}
		By Lemma~\ref{CHS3}, $f(a_0+\sum\limits_{i=1}^{k}a_i\alfa_i)=f_0(a_0)+ \sum\limits_{i=1}^{k}(a_if_0'(a_0)+f_i(a_0))\alfa_i$ 
		for all $a_0,\ldots,a_k\in R$. This immediately implies the ``if'' direction.
		To see the ``only if'', suppose that $f$ is a null polynomial on $\Ralfak$. 
		Then
		\begin{equation*}
		f_0(a_0)+ \sum\limits_{i=1}^{k}(a_if_0'(a_0)+f_i(a_0))\alfa_i=0 \text{ for all }  a_0,\ldots,a_k\in R.
		\end{equation*} 
		Clearly,  $f_0$ is a null polynomial on $R$. 
		Substituting first $0$, then $1$, for $a_i$, $i=1,\ldots,k$,  we find that $f_i$ and $f_0'$ 
		are null polynomials on $R$. Therefore $f_0\in \Nulld$ and $f_i\in\Null$ for $i=1,\ldots,k$. 
	\end{proof}
	Combining Lemma~\ref{CHS31} with Theorem~\ref{CHS4} gives the following criterion.
	\begin{corollary}\label{CHSNulleqc}
		Let $f =f_0 +\sum\limits_{i=1}^{k}f_i \alfa_i$, where $f_0,\dots, f_k \in R[x]$. 
		Then $f$ is a null polynomial on $\Ralfak$ if and only if 
		$f_0$ and $f_i\alfa_i$ are null polynomials on $\Ralfak$ for $i=1,\dots,k$.
	\end{corollary}
	Theorem~\ref{CHS4} implies the following corollary, which determines whether two polynomials $f,g\in{\Ralfak}[x]$ induce the same function on $\Ralfak$.  
	\begin{corollary}\label{CHS6} \label{CHSGencount}
		
		Let $f =f_0 +\sum\limits_{i=1}^{k}f_i \alfa_i$ and $g=g_0+\sum\limits_{i=1}^{k}g_i \alfa_i $, where 
		$f_0,\ldots, f_k, g_0,\ldots, g_k \in R[x]$.
		
		Then $f \quv g$ on $\Ralfak$ if and only if the following 
		conditions hold:
		\begin{enumerate}
			\item
			$[f_i]_R= [g_i]_R$ for $i=0,\dots,k$;
			\item
			$[f_0']_R = [g_0']_R$.
		\end{enumerate}
		
		In other words, $f \quv g$ on $\Ralfak$ if and only if
		the following  congruences hold:
		
		\begin{enumerate}
			
			\item
			$f_i \equiv g_i  \mod \Null$ for $i=1,\ldots,k$;
			\item
			$f_0 \equiv g_0  \mod \Nulld$.
		\end{enumerate}
	\end{corollary}
	\begin{proof}
		It is sufficient to consider the polynomial $h=f-g$ and notice that $f \quv g$ on $\Ralfak$ 
		if and only if $h \quv 0$ on $\Ralfak$. 
	\end{proof}
	\comment{
		\begin{corollary}
			The number of pairs of functions
			$(F,E)$ with $F\colon R\longrightarrow R$, $G\colon R\longrightarrow R$ 
			arising as $([f]_R, [f']_R)$ for some $f\in R[x]$ is equal to $\left[R[x]\colon \Nulld\right]$. 
		\end{corollary}
		\begin{proof}
			Set 
			\[\mathcal{A}=\{(F,E)\in \mathcal{F}(R)\times \mathcal{F}(R)\mid  \exists f\in R[x] \text{ such that }f,f'\text{ induce } F,E  \text{ respectively}\}.\]
			Define $\psi\colon R[x] \longrightarrow \mathcal{A}$ by $\psi(f)=([f]_R,[f']_R)$. 
			It is a routine verification to show that $\psi$ is a group epimorphism of additive groups with  $\ker\psi=N'_R$. Hence by the First Isomorphism Theorem 
			of groups, we get 
			$[R[x]\colon N'_R]=|\mathcal{A}|$. %A similar argument proves that  $|\PolFun|=\left[R[x]\colon \Null\right]$.
	\end{proof}} %end comment
	Recall from Definition~\ref{CHSequvfun} that $\PolFun[\Ralfak]$ denotes the set of polynomial functions on $\Ralfak$. In the following proposition, we derive a counting formula for $\PolFun[\Ralfak]$ depending on the indices of the ideals $\Null, \Nulld$. 
	\begin{proposition} \label{CHSfirstcountfor}
		The number of polynomial functions on $\Ralfak$ is given by
		\[|\PolFun[\Ralfak]|= \big[R[x]\colon \Nulld \big]\big[R[x]\colon \Null\big]^k.\]

	\end{proposition}
	
	\begin{proof}
		
		Let  $f =f_0 +\sum\limits_{i=1}^{k}f_i \alfa_i$ and $g=g_0+\sum\limits_{i=1}^{k}g_i \alfa_i $
		where $f_0,\ldots ,f_k,g_0,\ldots ,g_k \in R[x]$.  Then by Corollary~\ref{CHS6},
		$f  \quv g$  on $\Ralfak$ if and only if $f_0 \equiv g_0  \mod \Nulld$
		and $f_i \equiv g_i  \mod \Null$ for $i=1,\ldots,k$.
		
		Define $\varphi\colon \bigoplus\limits_{i=0}^{k} R[x] \longrightarrow \mathcal{F}(\Ralfak)$ 
		by $\varphi(f_0,\ldots,f_k)= [f]$, where $[f]$ is the function induced 
		on $\Ralfak$ by $f =f_0 +\sum\limits_{i=1}^{k}f_i \alfa_i$. 
		Then $\varphi$ is a group epimorphism of additive groups with  $\ker\varphi=\Nulld\times\bigoplus\limits_{i=1}^{k}  \Null$ by Theorem~\ref{CHS4}. 
		Hence \[|\mathcal{F}(\Ralfak)|=
		[\bigoplus\limits_{i=0}^{k} R[x]\colon \Nulld\times \bigoplus\limits_{i=1}^{k}  \Null]=
		[R[x]\colon \Nulld][R[x]\colon \Null]^k.\]

	\end{proof}
	The following proposition gives an upper bound for the minimal  degree of a representative of a polynomial function on $\Ralfak$.  
	\begin{proposition}\label{CHSsur}
		Let $h_1\in { \Ralfak}{[x]}$ and $ h_2 \in R[x]$ be  monic null polynomials on $\Ralfak$ and $R$, respectively, such that $\deg h_1= d_1$ 
		and  $\deg h_2=d_2$.
		
		Then every polynomial function 
		$F\colon \Ralfak\longrightarrow \Ralfak$ 
		is induced by a polynomial  $f =f_0 +\sum\limits_{i=1}^{k}f_i\alfa_i$, where $f_0,\ldots,f_k \in R[x]$ 
		such that $\deg f_0 <d_1$ and  $\deg   f_i < d_2$ for $i=1,\ldots,k$. 
		
		Moreover, if  $F$ is induced by a polynomial $f\in R[x]$ and 
		$h_1\in R[x]$ (rather than in ${\Ralfak}[x]$), then there 
		exists a polynomial $g\in R[x]$ with $\deg g<d_1$, such that 
		$[g]_R=[f]_R$ and $[g']_R=[f']_R$.
	\end{proposition}
	
	\begin{proof} Suppose that $h_1\in {\Ralfak}[x]$ is a monic null polynomial
		on $\Ralfak$ of degree $d_1$. Let $g \in {\Ralfak}[x]$ be a polynomial 
		that represents $F$. 
		By the division algorithm,  we have  $g(x) =q(x)h_1(x)+r(x)$ 
		for some $r,q \in {\Ralfak}[x]$, where $\deg r \le d_1 -1$. 
		Then clearly, $r(x)$ represents $F$. By Lemma~\ref{CHS21}, 
		$r=f_0+\sum\limits_{i=1}^{k}r_i\alfa_i$ 
		for some $f_0,r_1,\ldots,r_k\in R[x]$, and it is obvious that 
		$\deg f_0,\deg r_i \le d_1-1$ for $i=1,\dots,k$. 
		Now let $h_2\in R[x]$ be a monic null polynomial on $R$ of degree $d_2$.  
		Again,  by the division algorithm, we have for $i=1,\ldots,k$, $r_i(x) =q_i(x)h_2(x)+f_i(x)$ 
		for some $f_i ,q_i \in R[x]$, where $\deg f_i \le d_2 -1$.
		Then  by Corollary~\ref{CHS6},  $r_i\alfa_i \quv f_i\alfa_i$ on $\Ralfak$. 
		Thus $f =f_0 +\sum\limits_{i=1}^{k}f_i \alfa_i$ is the desired polynomial. \\
		For the second part, the existence of $g\in R[x]$ with $\deg g<d_1$ such that
		$f \quv g $ on $\Ralfak$ follows by the same argument given in the previous part.
		By Corollary~\ref{CHSGencount},   $[g]_R=[f]_R$ and $[g']_R=[f']_R$.   
	\end{proof}
	\begin{remark}\label{CHSexistmonicn}
		Let $h(x)=\prod\limits_{r\in R}(x-r)^2$. Then   $h$ is a monic polynomial in $R[x]$, and by Lemma~\ref{CHS31}, it is a null polynomial on $\Ralfak$. This shows that the monic null polynomial mentioned in the last part of Proposition~\ref{CHSsur} always exists.
	\end{remark}
	\section{ Permutation  polynomials on  $\Ralfak$}\label{Sec04}
	
	This section deals with the group of polynomial permutations on $\Ralfak$.  In the following theorem, we give a characterization for a permutation polynomial on  $\Ralfak$. 
	\begin{theorem}\label{CHSGenper} 
		Let  $R$ be a finite ring. Let $f =f_0+ \sum\limits_{i=1}^{k}f_i \alfa_i$, 
		where $f_0,\ldots,f_k \in R[x]$. Then $f$ is a permutation polynomial
		on $\Ralfak$ if and only if the following conditions hold:
		\begin{enumerate}
			\item $f_0$ is a permutation polynomial on $R$;
			\item for all $a\in R$,  $f_0'(a)$ is a unit in $R$.
		\end{enumerate}
		
	\end{theorem}
	
	\begin{proof} 
		$(\Rightarrow)$ 
		Let $c\in R$. Then $c\in \Ralfak$. Since $f$ is a 
		permutation polynomial on $\Ralfak $, there exist
		$a_0,\ldots,a_k \in R$ such that $f(a_0+ \sum\limits_{i=1}^{k}a_i \alfa_i)= c$.
		Thus,  by Lemma~\ref{CHS3}, $$f_0(a_0)+ \sum\limits_{i=1}^{k}(a_if_0'(a_0)+f_i(a_0))\alfa_i = c.$$
		So $f_0(a_0)=c$, therefore $f_0$ is onto, and hence a permutation polynomial on $R$. 
		
		Let $a\in R$ and suppose that $f_0'(a)$ is a non-unit in  $R$. 
		Then $f_0'(a)$ is a zerodivisor of $R$. 
		Let $b\in R$, $b\ne 0$, such that $bf_0'(a)=0$. 
		Then, by Lemma~ \ref{CHS02}, \[f(a+\sum\limits_{i=1}^{k}b\alfa_i)=f_0(a)+\sum\limits_{i=1}^{k}(bf_0'(a)+f_i(a))\alfa_i =f_0(a)+\sum\limits_{i=1}^{k}f_i(a)\alfa_i=f(a).\]
		So $f$ is not one-to-one, which is a contradiction. This proves (2).\\
		($\Leftarrow$) It is enough to show that $f$ is one-to-one. 
		Let $a_0,\ldots,a_k,b_0,\ldots,b_k  \in R$ such that 
		$$f(a_0+\sum\limits_{i=1}^{k}a_i\alfa_i)=f(b_0+\sum\limits_{i=1}^{k}b_i\alfa_i),$$  that is, \ 
		$$f_0(a_0)+\sum\limits_{i=1}^{k}(a_if_0'(a_0)+f_i(a_0))\alfa_i = f_0(b_0)+\sum\limits_{i=1}^{k}(b_if_0'(b_0)+f_i(b_0))\alfa_i$$ by Lemma~\ref{CHS02}.
		Then  we have $f_0(a_0)= f_0(b_0)$ and $a_if_0'(a_0)+f_i(a_0)= b_if_0'(b_0)+f_i(b_0)$ for $i=1,\ldots,k$.
		Hence $a_0= b_0$   since $f_0$ is a permutation polynomial on $R$. 
		Then, since $f_0'(a_0)$ is a unit in $R$,  $a_i=b_i$ follows for $i=1,\ldots,k$.  
	\end{proof}
	
	Theorem~\ref{CHSGenper} shows that the criterion 
	to be a permutation polynomial on $\Ralfak$ depends only
	on $f_0$, and  implies  the 
	following corollary.
	
	\begin{corollary}\label{CHSPPfirstcoordinate}
		Let $f =f_0 +\sum\limits_{i=1}^{k}f_i \alfa_i$,
		where $f_0,\dots,f_k \in R[x]$. 
		Then the following statements are equivalent:
		\begin{enumerate}
			\item $f$ is a permutation polynomial
			on $\Ralfak$; 
			\item $f_0+f_i\alfa_i$ is a permutation polynomial on $R[\alfa_i]$ for every $i\in \{1,\ldots,k\}$;
			\item $f_0$ is a permutation polynomial on $\Ralfak$;
			\item $f_0$ is a permutation polynomial on $R[\alfa_i]$ for every $i\in \{1,\ldots,k\}$.
		\end{enumerate}
	\end{corollary}
	Recall from definition~\ref{CHSequvfun}
	that    $\PrPol[A]$ stands for the group of polynomial permutations on the ring  $A$.
	\begin{corollary}
		The group $\PrPol[{R[\alfa_i]}]$ is embedded in $\PrPol[\Ralfak]$ for every $i= 1,\ldots,k$.
	\end{corollary}
	\begin{proof}
		Fix $i\in \{1,\ldots,k\}$ and let $F\in \PrPol[{R[\alfa_i]}]$. Then $F$ is induced by $f=f_0+f_i\alfa_i$ for some $f_0,f_i \in R[x]$. Furthermore, $f_0+f_i\alfa_i$ is permutation polynomial on $\Ralfak$ by Corollary~\ref{CHSPPfirstcoordinate}. 
		Define a function
		$\psi\colon \PrPol[{R[\alfa_i]}] \longrightarrow \PrPol[\Ralfak]$ by $\psi(F)=[f]_{\Ralfak} $, where $[f]_{\Ralfak}$ denotes the function induced by $f$ on $\Ralfak$. 
		By Corollary~\ref{CHSGencount}, $\psi$ is well defined and one-to-one. Now if $F_1\in \PrPol[{R[\alfa_i]}]$ is 
		induced by $g\in R[\alfa_i][x]$, then $f\circ g$ induces $F\circ F_1$ on $R[\alfa_i]$. Hence,
		\begin{align*}
		\psi(F\circ F_1) & =[f\circ g]_{\Ralfak}\\
		& =[f]_{\Ralfak} \circ[g]_{\Ralfak} \text{   since }f,g\in \Ralfak[R][x]\\
		& =\psi(F)\circ\psi(F_1). 
		\end{align*}
		This completes the proof. 
	\end{proof}
	
	\begin{remark}
		We will show in Proposition~\ref{CHSmovdreivat} that the condition on the derivative in Theorem~\ref{CHSGenper} is redundant, when $R$ is a direct sum of local rings none of which is a field.  
	\end{remark}
	\begin{lemma}\cite[Theorem~3]{Necha}\label{CHSMac}
		Let $R$ be a finite local ring with a maximal ideal $M\ne \{0\}$ and suppose that $f\in R[x]$. 
		Then $f$ is a permutation polynomial on $ R$ 
		if and only if the following conditions hold:
		\begin{enumerate}
			\item $f$ is a permutation polynomial on $R/M$;
			\item  for all $a\in R$, $f'(a)\ne 0\mod{M}$.
		\end{enumerate}
	\end{lemma}
	\begin{lemma}\label{CHSdirectper}
		Let $R$ be a finite ring  and suppose that $R=\oplus_{i=1}^{n}R_i$, where $R_i$ is local for $i=1,\dots,n$. Let $f=(f_1,\ldots,f_n)\in R[x]$, where $f_i\in R_i[x]$. Then $f$ is a permutation polynomial on $R$ if and only if  $f_i$ is a permutation polynomial on  $R_i$ for $i=1,\dots,n$.
	\end{lemma}
	\begin{proof}
		$(\Rightarrow)$  Suppose that $f$ is a permutation polynomial on $R$ and fix an $i$. Let $b_i \in R_i$. Then $(0,\dots,b_i,\ldots,0)\in R$. Thus  there exists $a=(a_1,\dots,a_i,\ldots,a_n)\in R$, where $a_j \in R_j$, $j=1,\dots,n$ such that $f(a)=(f_1(a_1),\ldots,f_i(a_i),\ldots,f_n(a_n))=(0,\dots,b_i,\dots,0)$. Hence $f_i(a_i)=b_i$, and therefore $f_i$ is surjective, whence $f_i$ is a permutation polynomial on $R_i$.
		
		$(\Leftarrow)$ Easy and left to the reader. 
	\end{proof}
	From now on, let $R^\times$ denote the group of units of $R$.
	\begin{proposition}\label{CHSmovdreivat}
		Let $R$ be a finite ring which is a direct sum of  local rings which are not fields, and let $f=f_0+\sum\limits_{i=1}^{k}f_i \alfa_i$, where $f_0,\ldots,f_k\in  R[x]$. Then $f$ is a permutation polynomial on $\Ralfak$ if and only if $f_0$ is a  permutation polynomial on $R$.
	\end{proposition}
	\begin{proof}
		($\Rightarrow$) Follows by Theorem~\ref{CHSGenper}.
		
		($\Leftarrow$) Assume that  $f_0$ is a permutation polynomial on $R$.   By Theorem~\ref{CHSGenper}, we need only show that $f_0'(r)\in R^\times$ for every $r\in R$. Write $f_0=(g_1,\ldots,g_n)$, where $g_i\in R_i[x]$ for $i=1,\ldots,n$. 
		Then $g_i$ is a permutation polynomial on $R_i$ for $i=1,\ldots,n$ by Lemma~\ref{CHSdirectper}. 
		Now let $r\in R$, so $r=(r_1,\ldots,r_n)$, where $r_i\in R_i$. 
		Hence $f'_0(r)=(g'_1(r_1),\ldots,g'_n(r_n))$ but $g'_i(r_i)\in R_i^{\times}$  by Lemma~\ref{CHSMac} 
		for $i=1,\dots,n$. Therefore $f'_0(r)=(g'_1(r_1),\ldots,g'_n(r_n))\in  R^{\times}$, i.e, $f'_0(r)$ is a unit in $R$ for every $r\in R$. Thus $f_0$ satisfies the conditions of Theorem~\ref{CHSGenper}. Therefore $f$ is a permutation polynomial on $\Ralfak$. 
	\end{proof}
	\begin{corollary}\label{CHSpersumnug}
		Let $R$ be a finite ring which is a direct sum of local rings which are not fields. Let $f\in R[x]$ be a permutation polynomial on $\Ralfak$. Then $f+h$ is a permutation polynomial on $\Ralfak$ for every $h\in \Null[R]$. In particular, $x+h$ is a permutation polynomial on $\Ralfak$ for every $h\in \Null[R]$.
	\end{corollary}
	%Recall from definition  that $\PrPol[\Ralfak]$  denotes the group of permutation polynomials on $\Ralfak$. 
	
	\begin{proposition} \label{CHSGeperncount}
		Let $R$ be a finite ring. 
		Let $B$ denote  the number of pairs of functions $(H,G)$ with
		\[H\colon R\longrightarrow R \text{ bijective  and }  G\colon R\longrightarrow R^\times\] 
		that occur as 
		$([g],[g'])$ for some $g\in R[x]$. 
		Then  the number  of polynomial permutations 
		on $\Ralfak$  is given by \[|\PrPol[\Ralfak]|=B\cdot |\PolFun[R]|^k.\] 
	\end{proposition}
	\begin{proof}
		Let $F\in \PrPol[\Ralfak]$. Then by definition   $F$ is induced by a polynomial $f =f_0 +\sum\limits_{i=1}^{k}f_i \alfa_i$, where  $f_0,\ldots,f_k\in R[x]$. By Theorem~\ref{CHSGenper}, \[[f_0]\colon R\longrightarrow R \text{ bijective, }  [f'_0]\colon R\longrightarrow R^\times \text{ and } [f_i]\text{ is arbitrary in }\PolFun \text{ for }i=1,\ldots,k. \] 
		The rest follows by Corollary~\ref{CHSGencount}. 
	\end{proof}
	
	In the next section, we show that the number $B$ of Proposition~\ref{CHSGeperncount} depends on the order of the pointwise stabilizer  of $R $ in the group $\PrPol[\Ralfak]$. However, when $R$ is a finite field, we can find explicitly this number. For this, we need the following well known lemma.
	\begin{lemma}\label{CHSPerf}
		Let $\mathbb{F}_q$ be a finite field with $q$ elements. 
		Then for all functions
		\[F,G\colon \mathbb{F}_q\longrightarrow \mathbb{F}_q, \]
		there exists $f\in\mathbb{F}_q[x]$ such that  
		\[(F, G)=([f] ,[f']) \text{ and } \deg f<2q.\] 
	\end{lemma}
	\begin{proof}
		Let $f_0,f_1\in\mathbb{F}_q[x]$ such that $[f_0] =F$ 
		and $[f_1] =G$ and set 
		\[f(x) = f_0(x) + (f'_0(x) - f_1(x))(x^q-x).\]
		Then \[f'(x) =(f''_0(x) - f'_1(x))(x^q-x)+f_1(x).\]
		Thus $[f]=[f_0]=F$ and  $[f']=[f_1]=G$ since $(x^q-x)$ is a null polynomial on $\mathbb{F}_q$. 
		Moreover, since $(x^q-x)$ is a null
		polynomial on $\mathbb{F}_q$, we can choose $f_0,f_1$ such that $\deg f_0,\deg f_1<q$. Hence $\deg f<2q$.  
	\end{proof}
	
	\begin{proposition}\label{CHS13.111}
		Let $\mathbb{F}_q$ be a finite  field with $q$ elements.  The number  of polynomial 
		permutations on $\mathbb{F}_q[\alfa_1,\ldots,\alfa_k]$  is given by 
		\[|\mathcal{P}(\mathbb{F}_q[\alfa_1,\ldots,\alfa_k])|=q!(q-1)^qq^{kq}.\]
	\end{proposition} 
	
	\begin{proof}
		Let $\mathcal{B}$ be the set of pairs of functions $(F,G)$ such that 
		\[F\colon \mathbb{F}_q\longrightarrow \mathbb{F}_q  \text{ bijective and } G\colon \mathbb{F}_q\longrightarrow \mathbb{F}_q\setminus\{0\}.\] By Lemma~\ref{CHSPerf}, each $(F,G)\in \mathcal{B}$ arises as $([f],[f'])$ for some $f\in \mathbb{F}_q[x]$.
		By Proposition~\ref{CHSGeperncount}, 
		$|\mathcal{P}(\mathbb{F}_q[\alfa_1,\ldots,\alfa_k])|=|\mathcal{B}|\cdot|\PolFun[\mathbb{F}_q]|^k$. Clearly $|\mathcal{B}|=q!(q-1)^q$ and $|\PolFun[\mathbb{F}_q]|^k=q^{kq}$. 
	\end{proof}
	
	\section{The stabilizer of $R$ in the group of polynomial
		permutations of $\Ralfak$}\label{CHSsc4}
	The main object of this section is to describe the order of the subgroup of those polynomial permutations on $\Ralfak$  that fix pointwise  each element of $R$, and then to use this order to find a counting formula for the number of polynomial permutations on $\Ralfak$.
	\begin{definition}\label{CHSstd}
		Let
		$\Stabk=\{F\in \PrPol[\Ralfak]\mid 
		F(a)=a \text{ for every } a\in R\}$. 
	\end{definition}
	Evidently, $\Stabk[R]$ is a subgroup of $\PrPol[\Ralfak]$.
	%	\begin{lemma}\label{CHSexnul}
	%	Let $f,g\in R[x]$ with $f\quv g$ on $R$. There exists $h \in \Null[R]$ such that $f=g+h$.
	%	\end{lemma}
	%	\begin{proof}
	%	Let $h=f-g$. Then $h$ has the desired property. 
	%\end{proof}
	\begin{proposition}\label{CHSfirststab}
		Let $R$ be a finite  ring. Then \[\Stabk=\{F\in \PrPol[\Ralfak]\mid F 
		\textnormal{ is induced by } x+h(x), h \in \Null[R] \}.\]
		In particular, every element of $\Stabk$ is induced by a polynomial in $R[x]$.
	\end{proposition}
	\begin{proof}
		It is obvious that 
		\[\Stabk[R]\supseteq\{F\in \PrPol[\Ralfak]\mid F 
		\textnormal{ is induced by } x+h(x), h \in \Null[R] \}.\]
		For the other inclusion, let $F\in \mathcal{P}(\Ralfak)$ 
		such that $F(a)=a$ for every $a\in R$. 
		Then $F$ is represented by $f_0+\sum\limits_{i=1}^{k}f_i\alfa_i$,  
		where $f_0,\ldots, f_k\in R[x]$, and  
		$a=F(a)=f_0(a)+\sum\limits_{i=1}^{k}f_i(a)\alfa_i$ for every $a\in R$. 
		It follows that $f_i(a)=0$ for every $a\in R$, i.e., 
		$f_i$ is a null polynomial on $R$ for $i=1,\dots,k$. 
		Thus  $f_0+\sum\limits_{i=1}^{k}f_i\alfa_i\quv f_0$ on $\Ralfak$
		by Corollary~\ref{CHSGencount}, that is, $F$ is represented by $f_0$. 
		Also, $f_0\quv id_{R}$ on $R$, where $id_{R}$ is the identity function on $R$, and  therefore  
		$f_0(x)=x+h(x)$ for some $h\in \Null$  by Remark~\ref{nulremark}. 
	\end{proof}
	We have the following theorem, when $R$ is a finite field, which describes the order of $\Stabk[\mathbb{F}_q]$. The proof
	is almost the same as in \cite[Theorem~4.11]{Haki}.
	\begin{theorem} \label{CHS1301}
		Let $\mathbb{F}_q$ be a finite  field with $q$ elements. Then: 
		\begin{enumerate}
			\item $|\Stabk[\mathbb{F}_q]| =|\{[f']_{\mathbb{F}_q}\mid  f\in \Null[\mathbb{F}_q] 
			\text{ and for every } a\in\mathbb{F}_q, 
			f'(a)\ne -1 \}|$;
			\item $|\Stabk[\mathbb{F}_q]|  =|\{[f']_{\mathbb{F}_q}\mid  f\in \Null[\mathbb{F}_q], 
			\deg f<2q \text{ and for every } a\in\mathbb{F}_q, 
			f'(a)\ne -1 \}|$;
			\item $|\Stabk[\mathbb{F}_q]|  =(q-1)^q$. \label{CHSst1}
		\end{enumerate} 
	\end{theorem}
	\begin{proof}
		We begin with the proof of (1) and (2). 
		Set  
		\[A=\{[f']_{\mathbb{F}_q}\mid  f\in \Null[\mathbb{F}_q] 
		\text{ and for every } a\in\mathbb{F}_q, 
		f'(a)\ne -1 \}.\]
		We define a bijection $\varphi$ from 
		$\Stabk[\mathbb{F}_q]$  to the set  $A$. 
		If $F\in \Stabk[\mathbb{F}_q]$, then it is represented by $x+h(x)$, where 
		$h\in \mathbb{F}_q[x]$ is a null polynomial  on $\mathbb{F}_q$,
		by Proposition~\ref{CHSfirststab}. 
		Now $h'(a)\ne -1 $  for every $a\in \mathbb{F}_q$,
		by Theorem~\ref{CHSGenper}, 
		whence $[h']_{\mathbb{F}_q}\in A$.  Now,   set $\varphi(F)=[h']_{\mathbb{F}_q}$. Then 
		Corollary~\ref{CHS6} shows that $\varphi$ is well-defined and injective. To show $\varphi$ is surjective, let
		$[h']_{\mathbb{F}_q}\in A$, where $h\in \Null[\mathbb{F}_q]$. Then,
		by Theorem~\ref{CHSGenper} and Proposition~\ref{CHSfirststab}, $F=[x+h]_{\Ralfak[\mathbb{F}_q]}\in \Stabk[\mathbb{F}_q]$.
		Thus $\varphi(F) =[h']_{\mathbb{F}_q}$.
		Moreover, by Lemma~\ref{CHSPerf}, $h$ can be chosen such that 
		$\deg h<2q$. %since $(x^q-x)^2$ is a null polynomial on $\mathbb{F}_q$.
		
		Next, we prove (3).  By (1), 
		%$$|\Stabk[\mathbb{F}_q]| =|\{[f']_{\mathbb{F}_q}: f\in \Null[\mathbb{F}_q] 
		%\text{ and for every } a\in\mathbb{F}_q,  f'(a)\ne -1 \}|.$$  
		$$|\Stabk[\mathbb{F}_q]|\le|\{G \colon \mathbb{F}_q\longrightarrow
		\mathbb{F}_q\setminus\{-1\} \}|= (q-1)^q.$$
		Now for every function $G \colon \mathbb{F}_q\longrightarrow \mathbb{F}_q\setminus\{-1\}$ there exists a polynomial
		$f\in \Null[\mathbb{F}_q]$ such that $[f']_{\mathbb{F}_q}=G$ by Lemma~\ref{CHSPerf}. Thus  $f(x)+x$ is a permutation polynomial on $\Ralfak[\mathbb{F}_q]$ by Theorem~\ref{CHSGenper}. Obviously, $x+f(x)$ induces the
		identity on $\mathbb{F}_q$, and hence $[x+f(x)]_{\Ralfak[\mathbb{F}_q]}\in \Stabk[\mathbb{F}_q]$. Therefore every element of the set $\{G \colon \mathbb{F}_q\longrightarrow
		\mathbb{F}_q\setminus\{-1\} \}$ corresponds to an element of  $\Stabk[\mathbb{F}_q]$, from which we conclude that $| \Stabk[\mathbb{F}_q]|\ge (q-1)^q$. 
		This completes the proof. 
	\end{proof}
	\begin{notation}
		
		Let 	$$
		\mathcal{P}_R(\Ralfak)=\{F\in \PrPol[\Ralfak]\mid  F=[f]_{\Ralfak} \textnormal{ for some }
		f\in R[x]\}
		.$$
		In similar manner, let
		$\mathcal{P}_R(R[\alfa_i])=\{F\in \PrPol[{R[\alfa_i]}]\mid  F=[f]_{R[\alfa_i]} \textnormal{ for some }
		f\in R[x]\}.$ 
	\end{notation}
	We now show that $\mathcal{P}_R(\Ralfak)$ is a subgroup of $\PrPol[\Ralfak]$. 
	\begin{proposition}\label{CHSkthrelation}
		The set $\mathcal{P}_R(\Ralfak)$ is a subgroup of $\PrPol[\Ralfak]$ and\\
		$\mathcal{P}_R(\Ralfak)\cong \mathcal{P}_R(R[\alfa_i])$ for $i=1,\ldots,k$.
	\end{proposition}
	\begin{proof}
		It is clear that $\mathcal{P}_R(\Ralfak)$ is closed under composition. Since it is finite, it is a subgroup of $\PrPol[\Ralfak]$. Let $F\in \mathcal{P}_R(\Ralfak)$ and suppose that $F$ is induced by  $f\in R[x]$. 
		Define \[
		\psi\colon \mathcal{P}_R(\Ralfak) \longrightarrow \mathcal{P}_R(R[\alfa_i]),\quad
		F\mapsto [f]_{R[\alfa_i]}.
		\]
		Then $\psi$ is well defined by Corollary~\ref{CHSGencount}, and evidently it is a homomorphism. By Corollary~\ref{CHSPPfirstcoordinate}, $\psi$ is surjective. To show that $\psi$ is one-to-one, let $F_1\in \mathcal{P}_R(\Ralfak)$ be induced by $g\in R[x]$ with $F\ne F_1$. Then either $f\nquv g$ on $R$ or $f'\nquv g'$ on $R$ by Corollary~\ref{CHSGencount}. Thus $\psi(F)=[f]_{R[\alfa_i]}\ne \psi(F_1)=[g]_{R[\alfa_i]}$. 
	\end{proof}
	We will see  that $\Stabk[R]$ is a normal subgroup of $\mathcal{P}_R(\Ralfak)$. But first we prove the following
	fact.
	\begin{proposition} \label{CHSperun}
		Let $R$ be a finite ring. Then   for every $F\in \mathcal{P}(R)$ there exists a polynomial $f\in  R[x]$ such that $F$ is induced by $f$ and $f'(r)\in R^{\times}$ for every $r\in R$. 
	\end{proposition}
	\begin{proof}
		Set $\mathcal{P}_u(R)=\{F \in \mathcal{P}(R)\mid  F\text{ is induced by }f\in R[x], f'\colon R\longrightarrow R^{\times} \}$.
		By definition $\mathcal{P}_u(R)\subseteq\mathcal{P}(R)$. 
		Let $F \in  \mathcal{P}(R)$. Then $F$ is induced by $f\in R[x]$.  
		Since $R$ is finite, $R=\oplus_{i=1}^{n}R_i$, where $R_i$ are local rings. We distinguish two cases. For the first case, we  suppose that  no $R_i$ is a field. Then   $f$ is a permutation polynomial on $\Ralfak$ by Proposition~\ref{CHSmovdreivat}.
		Hence $f'(a)\in R^{\times}$ for every $a\in R$ by Theorem~\ref{CHSGenper}.  So $F\in \mathcal{P}_u(R)$.
		For the second case, we assume without loss of generality   that   $R_1,\ldots,R_r$ are fields and no  $R_{i}
		$ is a field for  $i>r$. We identify $R[x]$ with $\bigoplus\limits_{i=1}^{n} R[x]$ and  write $f=(f_1,\ldots,f_n)$ where $f_i\in R_i[x]$ for $i=1,\ldots,n$. By Lemma~\ref{CHSdirectper}, $f _i $ is a permutation polynomial on $  R_i$, for $i=1,\dots,n$.  Now  a similar argument like the one given in the first case  shows that $f'_i(a_i)\in R_i^{\times}$ for every $a_i\in R_i$ for $i=r+1,\dots,n$.
		On the other hand, there exists $g_j\in R_j[x]$ such that $g_j \quv f_j$ on $R_j$ and
		$g'_j(a_j)\in R_j^{\times}$ for every
		$a_j\in R_j$, $ j=1,\ldots,r$ by Lemma~\ref{CHSPerf}.
		Then take $g=(g_1,\dots,g_r,f_{r+1},\ldots,f_n)$. Thus $g\quv f$ on $R$ and $g'(a)\in R^{\times}$ for every $a\in R$. Therefore $g$ induces $F$ and $F\in \mathcal{P}_u(R)$. 
	\end{proof} 
	\begin{theorem}\label{CHSPZlemma}
		Let $R$ be a finite ring. Then:
		%Let $\PZPnlafa$ as above, $\PrPol[\mathbb{Z}_{p^n}]$ the group of
		%polynomial permutations on $\mathbb{Z}_{p^n}$ and
		%$\Stab[p^n]$ as in Definition~\ref{CHSstd}. 
		%\[
		%\varphi:\mathcal{P}_R(\Ralfak) \longrightarrow \mathcal{P}_u(R),\quad
		%F\mapsto F_{\big|R},
		%\]
		
		%where $F_{\big|R}$ denotes the
		%restriction of the function $F$ to $R$,
		%is a group epimorphism with $\ker\varphi = Stab_{\alfa_1,\ldots,\alfa_k}(R)$.
		%In particular:
		
		\begin{enumerate}
			\item\label{CHSfirstPZ}
			every element of $\mathcal{P}(R)$ occurs as
			the restriction to $R$ of some $F\in \mathcal{P}_R(\Ralfak)$;
			\item\label{CHSscndPZ}
			$\mathcal{P}_R(\Ralfak)$ contains
			$\Stabk[R] $ as a normal subgroup and 
			\[
			\raise2pt\hbox{$\mathcal{P}_R(\Ralfak)$} \big/ \lower2pt\hbox{$\Stabk[R]$}
			\cong \mathcal{P}(R).
			\]
		\end{enumerate}
	\end{theorem}
	
	\begin{proof}
		
		(\ref{CHSfirstPZ})  This is obvious from 
		Proposition~\ref{CHSperun} and Theorem~\ref{CHSGenper}.
		
		(\ref{CHSscndPZ}) $\Stabk[R]$ is contained in $\mathcal{P}_R(\Ralfak)$,
		because every element of
		$\Stabk[R]$ can be represented by a polynomial with  coefficients in $R$
		by Proposition~\ref{CHSfirststab}.
		Let $F\in \mathcal{P}_R(\Ralfak)$ be represented by  $f\in R[x]$. Then define  
		$\varphi\colon \mathcal{P}_R(\Ralfak) \longrightarrow \mathcal{P}(R)$
		by $\varphi(F)= [f]_R$. Now $\varphi$ is well defined by Corollary~\ref{CHSGencount}, and it is 
		a group homomorphism with $\ker\varphi = \Stabk[R]$.
		By Proposition~\ref{CHSperun}, $\varphi$ is surjective. 
	\end{proof}
	\begin{corollary} \label{CHSmshi}
		
		For any fixed $F\in \PrPol[R]$,
		\[
		\left|\Stabk[R]\right|=\left|\{[f']_R\mid 
		f\in R[x], [f]\in \mathcal{P}_R(\Ralfak) \text{ and } 
		[f]_ R= F \}\right|.
		\]
		
	\end{corollary}
	
	\begin{proof}
		Let $f\in R[x]$ be a permutation polynomial on $\Ralfak$ with 
		$[f]_R=F$. Such an $f$ exists by Theorem~\ref{CHSPZlemma} (\ref{CHSfirstPZ}).
		We denote by $[f]$ the permutation induced
		by $f$ on $\Ralfak$. 
		Then the coset of $[f]$ with respect to $\Stabk[R]$ has $|\Stabk[R]|$
		elements. By Theorem~\ref{CHSPZlemma} (\ref{CHSscndPZ}), this coset consists of 
		all polynomial permutations $G\in \mathcal{P}_R(\Ralfak)$ with 
		$[f]_R=G_{\big|R}$, where $G_{\big|R}$ is the restriction of the function $G$ to $R$.
		Let $g\in R[x]$ with $[g]=G$.
		By Corollary~\ref{CHSGencount}, $G\ne [f]$ if and only if 
		$[f']_R\ne [g']_R$. 
		Thus we have a bijection between the coset of
		$[f]$ with respect to $\Stabk[R]$ and the set of functions 
		$[g']_R$ occurring for $g\in R[x]$ such that
		$[g]=G$ permutes $\Ralfak$ and 
		$[f]_R=[g]_R$. 
	\end{proof}
	
	When $R$ is a finite ring which is a direct sum of local rings that are not fields,  Corollary~\ref{CHSmshi} is
	a special case of a general result (see Proposition~\ref{CHS14c}).
	
	We now employ Corollary~\ref{CHSmshi} to find the number of  polynomial permutations on $\Ralfak$ in terms of $|\Stabk[R]|$.
	\begin{theorem}\label{CHS14} 
		Let $R$ be a finite ring.	For any integer $k\ge 1$,
		\[|\PrPol[\Ralfak]|=|\PolFun[{R}]|^k\cdot
		|\PrPol[R]|\cdot |\Stabk[R]|.\]
		
	\end{theorem}
	
	\begin{proof}
		For $f\in R[x]$, let $[f]$ be the function induced by $f$ on $\Ralfak$.\\
		
		Set 
		$
		B = \bigcup\limits_{\rlap{$\scriptstyle{F\in \PrPol[R]}$}}
		\{[f']_R\mid  f\in R[x], [f]\in \mathcal{P}_R(\Ralfak) \text{ and } 
		[f]_R= F \}$.
		
		Then $|B|=|\PrPol[R]|\cdot |\Stabk[R]|$ 
		by Corollary~\ref{CHSmshi}.\\ 
		Now we define a function 
		$ \Psi\colon \PrPol[\Ralfak] \longrightarrow B\times \prod\limits_{i=1}^{k}  \PolFun[{R}]$ 
		as  follows: if $G\in\PrPol[\Ralfak]$  is induced by 
		$g=g_0+\sum\limits_{i=1}^{k}g_i\alfa_i$, where $g_0,\ldots,g_k \in R[x]$,
		we let $\Psi(G)=([g'_0]_R,[g_1]_R,\ldots,[g_k]_R)$. By Theorem~\ref{CHSGenper} and Corollary~\ref{CHS6}, $\Psi$ is
		well-defined and one-to-one. The surjectivity of 
		$\Psi$  follows   by  
		Theorem~\ref{CHSPZlemma} and Theorem~\ref{CHSGenper}. % Theorem~\ref{CHS11}, and for the case $n=1$ by Lemma~\ref{CHSPerf} and Theorem~\ref{CHS11a}. 
		Therefore
		\[|\PrPol[\Ralfak]|=|B\times\prod\limits_{i=1}^{k} \PolFun[R]|= 
		|\PrPol[R]|\cdot |\Stabk[R]|\cdot| \PolFun[R]|^k.\]  \end{proof} 
	
	\begin{definition} \label{CHS11.12}
		
		Let
		$\N[n]{R}=\{f\in R[x]\mid  f\in \Null[R]
		\textnormal{ with }\deg f < n\}$, and
		\[
		\Nd[n]{R}=\{f\in R[x]\mid 
		f\in \Nulld[R] \textnormal{ with }\deg f <n\}.\]
	\end{definition}

	In the following theorem, we obtain several descriptions for the order of the group $\Stabk[R]$ whenever $R$ is a direct sum of local rings which are not fields.
	\begin{theorem}\label{CHS12} 
		Let $R$ be a finite ring which is a direct sum of local rings that are not fields. Then the following hold.
		\begin{enumerate}
			\item\label{CHSscndstab}
			$
			|\Stabk[R]|  =|\{[f']_R\mid  f\in \Null[R] 
			\}|.
			$
			\item\label{CHSthirdstab}
			If there exists a monic null polynomial on $\Ralfak$ in $R[x]$ of degree~$n$, 
			then:
			\begin{enumerate}

				\item \label{CHSthirdstaba}
				
				$
				|\Stabk[R]|  =
				|\{[f']_R\mid  f\in \Null[R]  \textnormal{ with }\deg f<n\}|;
				$
				\item\label{CHSfourthstab} $
				|\Stabk[R]|  = [\Null[R] \colon \Nulld[R]]=
				\frac{|\N[n]{R}|}{|\Nd[n]{R}|}.
				$
			\end{enumerate}
		\end{enumerate}
		
	\end{theorem}
	
	\begin{proof}
		(\ref{CHSscndstab})
		We define a bijection $\varphi$ from 
		$\Stabk[R]$  to the set of  functions induced on
		$R$ by the   derivative of some null polynomial on $R$. 
		By Proposition~\ref{CHSfirststab},
		every $F\in \Stabk[R]$ is represented by $x+f(x)$, where $f\in R[x]$ 
		is a null polynomial on $R$. We set $\varphi(F)=[f']_R$. 
		Then Corollary~\ref{CHS6} shows that $\varphi$ is well-defined and injective,
		and Corollary~\ref{CHSpersumnug} shows that  it is surjective. \\
		(2) Monic null polynomials on  $\Ralfak$ with coefficients in $ R$ always exist by Remark~\ref{CHSexistmonicn}.
		
		(\ref{CHSthirdstaba}) If $g\in \Null[R]$, then by 
		Proposition~\ref{CHSsur}, there exists 
		$f\in R[x]$ with $\deg f<n$ such that 
		$[f]_R=[g]_R$ and $[f']_R=[g']_R$. Evidently, $f\in \Null[R]$.
		
		(\ref{CHSfourthstab})
		For computing the index, define $\varphi\colon \Null[R] \longrightarrow \PolFun[R]$
		by $\varphi(f)=[f']_R$.
		Clearly, $\varphi$ is a   homomorphism of additive groups. 
		Furthermore, 
		\[\ker\varphi=\Nulld[R] \textnormal{ and } 
		\im\varphi=\{[f']_R\mid  f\in \Null[R]\}, 
		\] and hence $\raise1.5pt\hbox{$\Null[R]$} \big/ \lower1.5pt\hbox{$\Nulld[R]$} \cong \{[f']_R\mid  f\in \Null[R]\}$. Therefore $|\Stabk[R]|=[\Null[R]\colon \Nulld[R]]$ by (\ref{CHSscndstab}).\\
		Finally consider the ratio, consider the sets $\N[n]{R}$ and $\Nd[n]{R}$ as
		defined in Definition~\ref{CHS11.12}.  The equivalence relation in 
		Definition~\ref{CHSequvfun} restricted to these two additive subgroups 
		and the analogous proof to the previous part show  that 
		\[|\Stabk[R]|=[\N[n]{R} \colon \Nd[n]{R}].\] 
	\end{proof}
	
	\begin{remark}
		\leavevmode
		\begin{enumerate}
			\item When $R=\mathbb{F}_q$  is a finite field,  we have shown  in Theorem~\ref{CHS1301}~(\ref{CHSst1}) that 
			$|\Stabk[\mathbb{F}_q]|=(q-1)!$. But we will see later that   \[[\Null[\mathbb{F}_q] \colon  \Nulld[\mathbb{F}_q]]=[\N[2q]{\mathbb{F}_q} \colon \Nd[2q]{\mathbb{F}_q}]=q^q.\]
			\item Even when $k=1$,    Theorem~\ref{CHS12} is still a generalization of \cite[Proposition~7.2]{Haki}.
			%\item In~\cite{Onpro}, the author proved that $St_{\alfa}(R)\cong \Null /\Nulld$ for a wide class of finite local rings.
			
		\end{enumerate}
		
	\end{remark}
	
	\begin{proposition}\label{CHS14c} 
		Let $R$ be a finite ring which is a direct sum of local rings that are 
		not fields. Then for any fixed $F\in \PolFun [R]$,
		\[
		|\Stabk[R]|=|\{[g']_{R}\mid 
		g\in R[x] \text{ with\/ } [g]_R=F \}|.
		\]
	\end{proposition}
	
	\begin{proof}
		Set
		\[A=\{ [g']_R \mid 
		g\in R[x] \text{ with\/ } [g]_R=F \},\] and
		fix $g_0\in R[x]$ with $[g_0]_R=F $. Then
		$g-g_0$ is a null polynomial on $R$
		for any $g\in R[x]$ with $  [g']_R\in A$.
		
		We define a bijection
		\[ \phi\colon A \longrightarrow \{[f']_R\mid 
		f\in \Null[R]\},\quad
		\phi([g']_R)=[(g-g_0)']_R .\]
		Since $[(g-g_0)']_R = [g']_R - [g_0']_R$,
		$\phi$ is well defined. 
		Further, $\phi$ is injective, because, for two distinct elements of $A$,
		$[g_1']_R\ne [g']_R$ implies
		that
		$[(g_1-g_0)']_R\ne [(g-g_0)']_R$.
		
		Now, consider
		$[f']_R$, where $f\in \Null[R]$.
		Then  $[g_0+f]_R=F$ and, thus,
		$[g'_0+f']_R$ is in $A$ and $\phi([g'_0+f']_R)=[f']_R $. 
		Therefore $\phi$ is surjective.
		
		By Theorem~\ref{CHS12}~(\ref{CHSscndstab}),
		\[
		|\Stabk[R]| = |\{[h']_R\mid  h\in \Null[R]\}| = |A|.
		\] 
	\end{proof} 
	
	\begin{remark}
		\leavevmode
		\begin{enumerate} 
			\item
			For any fixed polynomial function $F\in \PolFun$, Proposition~\ref{CHS14c} tells us that the cardinality of the    set
			\[ \{[g']_R\mid 
			g\in R[x] \text{ with\/ } [g]_R=F \}\]
			is independent from our choice of the polynomial function $F$.
			\item Note that Proposition~\ref{CHS14c} is a generalization of  \cite[Corollary~7.6]{Haki}
			which considers the case when   $k=1$   and $R=\mathbb{Z}_{p^n}$ ($n>1$).
		\end{enumerate}
	\end{remark}
	Next we show that   for all $k\ge 1$ the stabilizer groups $\Stabk[R]$ are isomorphic.
	\begin{theorem}\label{CHSstabiso}
		Let $R$ be a finite ring and let $k$ be a positive integer. Then
		$\Stabk[R] \cong St_{\alfa_i}(R)$ for $i=1,\ldots,k$.
	\end{theorem}
	\begin{proof}
		Fix   $i\in \{1,\ldots,k\}$. Then by the definition of dual numbers (for the case $k=1$), $R[\alfa_1]\cong R[\alfa_i]$. % Thus  by Theorem~\ref{CHS12}~(\ref{CHSfourthstab}),  $|St_{\alfa_1}(R)|=|St_{\alfa_i}(R)|=[\Null[R]:\Nulld[R]]=|\Stabk[R]|$.\\
		Let $F\in \mathcal{P}_R(\Ralfak)$ and suppose that $F$ is induced by  $f\in R[x]$. 
		Define \[
		\psi\colon \mathcal{P}_R(\Ralfak) \longrightarrow \mathcal{P}_R(R[\alfa_i]),\quad
		F\mapsto [f]_{R[\alfa_i]}.
		\]
		The proof of Proposition~\ref{CHSkthrelation} shows that $\psi$ is an isomorphism. If $\phi$ denotes the restriction of $\psi$ to $\Stabk[R]$, then $\Stabk[R]\cong \phi(\Stabk[R])$.   % $|\Stabk[R]|=|St_{\alfa_i}(R)|$
		Therefore, we need only  show that $\phi(\Stabk[R])= St_{\alfa_i}(R) $. Let $G\in St_{\alfa_i}(R)$. Then $G$ is induced by $x+h(x)$ for some $h\in \Null$ by Proposition~\ref{CHSfirststab} (with $k=1$). By Corollary~\ref{CHSPPfirstcoordinate} and Proposition~\ref{CHSfirststab}, $F=[x+h(x)]_{\Ralfak}\in \Stabk$.
		But then $\phi(F)=\psi(F)=[x+h(x)]_{R[\alfa_i]}=G$, hence $G\in \phi(\Stabk[R])$.  This shows that 
		$St_{\alfa_i}(R)\subseteq \phi(\Stabk[R])$. The other inclusion is similar.  
	\end{proof}
	We need the following lemma which is a straightforward result of the Third Isomorpism Theorem for rings. 
	\begin{lemma}\label{CHS2ndisoapli}
		Let $R$ be a finite ring. Then $[R[x]\colon \Nulld[R]]=[R[x] \colon \Null[R]][\Null[R] \colon \Nulld[R]]$.
	\end{lemma}
	
	\begin{theorem}\label{CHSCounttheor}
		Let $R$ be a finite ring. Then \[|\PolFun[\Ralfak]|=[\Null[R] \colon \Nulld[R]] |\PolFun[R]|^{k+1}.\]  Moreover, when $R$ is a direct sum of local rings which are not fields, we have
		\[|\PolFun[\Ralfak]|=|\Stabk[R]|\cdot|\PolFun[R]|^{k+1}.\] 
	\end{theorem}
	\begin{proof}
		We have,
		\begin{align*}
		|\PolFun[\Ralfak]| &=[R[x] \colon \Nulld[R]] |\PolFun[R]|^{k} \text{ (By Proposition~\ref{CHSfirstcountfor} and Remark~\ref{nulremark})}\\
		&=[\Null[R] \colon \Nulld[R]]|\PolFun[R]|^{k+1} \text{ (By Lemma~\ref{CHS2ndisoapli})}.
		\end{align*}
		The second part follows from the above and Theorem~\ref{CHS12}~(\ref{CHSfourthstab}).      
	\end{proof}
	We turn now to find explicitly the number of polynomial functions on $\Ralfak[\mathbb{F}_q]$. To do this, we need 
	the following  lemma, and we leave its proof to the reader.
	\begin{lemma}\label{CHSNidFid} 
		Let $\mathbb{F}_q$ be a finite field. Then:
		\begin{enumerate}
			\item $\Null[\mathbb{F}_q]=(x^q-x)\mathbb{F}_q[x]$;
			\item $\Nulld[\mathbb{F}_q]=(x^q-x)^2\mathbb{F}_q[x]$.
		\end{enumerate}
	\end{lemma}
	
	\begin{proposition}\label{CHScontfld}
		Let $\mathbb{F}_q$ be a finite field. Then $|\PolFun[{\Ralfak[\mathbb{F}_q]}]|=q^{(k+2)q}$.
	\end{proposition}
	\begin{proof}
		Set
		\[ \mathcal{A}=\{f\mid  f=f_0+\sum\limits_{i=1}^{k}f_i\alfa_i, \text{ where } f_0,f_i \in \mathbb{F}_q[x],  \deg f_0<2q, \deg f_i<q \text{ for }i=1,\ldots,k\}.\]
		Then it is clear that $|\mathcal{A}|=q^{(k+2)q}$. To complete the proof, we show that if $f,g \in \mathcal{A}$ with $f\ne g$, then $[f]\ne [g]$, or equivalently if $[f]=[g]$, then $f=g$.
		Suppose that $f,g \in \mathcal{A}$, where $f_0+\sum\limits_{i=1}^{k}f_i\alfa_i$ and $g_0+\sum\limits_{i=1}^{k}g_i\alfa_i$,  such that $[f]=[g]$. 
		Thus $[f-g]$ is the zero function on $\Ralfak[\mathbb{F}_q]$. Hence $f-g= (f_0-g_0)+\sum\limits_{i=1}^{k}(f_i-g_i)\alfa_i$ is a null polynomial on $\Ralfak[\mathbb{F}_q]$, whence $f_0-g_0\in \Nulld[\mathbb{F}_q]$ and $f_i-g_i\in \Null[\mathbb{F}_q]$ for $i=1,\ldots,k$ by 
		Theorem~\ref{CHS4}. Then, by Lemma~\ref{CHSNidFid}, we have $(x^q-x)^2\mid (f_0-g_0)$ and  $(x^q-x)\mid (f_i-g_i)$ for $i=1,\ldots,k$. Therefore 
		$f_0-g_0=0$, $f_i-g_i=0$ for $i=1,\ldots,k$ since $\deg (f_0-g_0)<2q$ and $\deg (f_i-g_i)<q$ for $i=1,\ldots,k$. Thus $f=g$. 
	\end{proof}
	The following corollary shows that, when $R=\mathbb{F}_q$, $[\Null[\mathbb{F}_q]\colon \Nulld[\mathbb{F}_q]]\ne |\Stabk[\mathbb{F}_q]|$(see Theorem~\ref{CHS1301} and Theorem~\ref{CHS12}).
	\begin{corollary}
		Let $\mathbb{F}_q$ be a finite field. Then $[\Null[\mathbb{F}_q] \colon \Nulld[\mathbb{F}_q]]=[\N[2q]{\mathbb{F}_q} \colon \Nd[2q]{\mathbb{F}_q}]=q^q$.
	\end{corollary}
	\begin{proof}
		By Theorem~\ref{CHSCounttheor}, $|\PolFun[{\Ralfak[\mathbb{F}_q]}]|=[\Null[\mathbb{F}_q] \colon \Nulld[\mathbb{F}_q]] |\PolFun[\mathbb{F}_q]|^{k+1}$, whence $[\Null[\mathbb{F}_q] \colon \Nulld[\mathbb{F}_q]]=q^q$ by   Proposition~\ref{CHScontfld}.
		On the other hand, Lemma~\ref{CHSNidFid} gives   $|\N[2q]{\mathbb{F}_q}|=q^q$ and  $|\Nd[2q]{\mathbb{F}_q}|=1$. Thus \[[\N[2q]{\mathbb{F}_q}\colon \Nd[2q]{\mathbb{F}_q}]=\frac{|\N[2q]{\mathbb{F}_q}|}{    |\Nd[2q]{\mathbb{F}_q}|}=q^q.\] 
	\end{proof}
	
	\section{Necessary and sufficient conditions}\label{CHSsec5}
	In this section, an algorithm is provided which decides whether or not a given function on
	$R[\alpha_1,\ldots,\alpha_k]$ is a polynomial function and if this is so it returns its polynomial representation.

	Motivated by ~\cite[Theorem 5]{moti}, we prove the following theorem.
	
	\begin{theorem}\label{CHSsystemlinear} Let  $R$ be a finite commutative ring with $n$ elements, and let $d_1,d_2$ be as in Proposition~\ref{CHSsur}.
		Let $F\colon R[\alpha_1,\ldots,\alpha_k]\longrightarrow R[\alpha_1,\ldots,\alpha_k]$ be a function and, for $0\le i\le k$, $F_i\colon R^{k+1}\longrightarrow R$   the functions such that  $$F(r_{ 0}+\sum\limits_{i=1}^{k}r_{i}\alpha_i)=F_0(r_0,\ldots,r_k)  +\sum\limits_{i=1}^{k}F_i(r_0,\ldots,r_k)\alpha_i,$$   		
		for  all  $(r_0,\ldots,r_{k})\in R^{k+1}$. Then the following statements are equivalent:
		
		\begin{enumerate}
			\item $F$ is a polynomial function on $R[\alpha_1,\ldots,\alpha_k]$;
			\item $F$ can be represented by a polynomial of degree $ \le d_1-1$;
			\item $F$ can be represented by a polynomial
			\[f(x) = f_0(x) +\sum\limits_{i=1}^{k}
			f_i(x)\alpha_i,\]
			%	\[f(x)= \sum\limits_{l=0}^{d_1-1}a_{0\,l}x^l+\sum\limits_{i=1}^{k}
			%\sum\limits_{m=0}^{ d_2-1}a_{i\,m}x^m\alpha_i,\]
			where  $f_0(x)=\sum\limits_{l=0}^{d_1-1}a_{0\,l}x^l$, $f_i(x)=\sum\limits_{m=0}^{ d_2-1}a_{i\,m}x^m $ with $a_{0\,l},a_{i\,m}\in R$ for $l=0,\ldots,d_1-1$,\ $i=1,\ldots,k, \, m=0,\ldots,d_2-1$;

			\item $F_0(r_0,\ldots,r_k)$ depends only on $r_0$; and for $1\le i \le k$, $F_i(r_0,\ldots,r_k)$ depends only on $r_0$ and $r_i$.
			
			The system of $|R|+k{|R|^{2}}$ linear equations in $d_1+kd_2$ variables ($y_{0\,l}$
			with $0\le l< d_1$ and  $y_{i\,m}$ with $1\le i\le k$ and $0\le m< d_2$), 
			\begin{align}\label{CHSSys}
			\sum\limits_{l=0}^{d_1-1}y_{0\,l} r_{0}^l& =  b_0(r_0)\nonumber
			\\
			\sum\limits_{l=1}^{d_1-1}(ly_{0\,l}r_{0}^{l-1})r_{i} +\sum\limits_{m=0}^{d_2-1}y_{i\,m} r_{0}^m & = b_i(r_0,r_i) \text{ for
			} i=1,\ldots,k,   
			\end{align}
			where $b_0(r_0)=F_0(r_0,\ldots,r_k)$, $b_i(r_0,r_i)=F_i(r_0,\ldots,r_k)$ and  $r_j$ varies through all elements of $R$  for $j=0,1,\dots,k$,	
			has a solution in $R$.
			%$y_{0\,l}=a_{0\,l},y_{i\,m}=a_{i\,m} \text{ with }  a_{0\,l},a_{i\,m}\in R$ 
			%for $l=0,\ldots,d_1 -1$; $i=1,\dots,k$; $m=0,\ldots,d_2 -1$.
		\end{enumerate}
	\end{theorem}
	\begin{proof}It is clear that $(3)\Rightarrow(2)\Rightarrow(1)$. The implication  $(1)\Rightarrow(3)$ follows by    Proposition~\ref{CHSsur}. 
		$(3)\Rightarrow(4)$  the first statement follows from Corollary~\ref{CHSneccondit}. Then, suppose that $F$ can be represented by a polynomial $f\in R[\alfa_1,\ldots,\alfa_k][x]$, where  
		$f= f_0+\sum\limits_{i=1}^{k}f_i\alpha_i, \text{ such that } f_0(x)=\sum\limits_{l=0}^{d_1-1}a_{0\,l}x^l, 
		f_i(x)=\sum\limits_{m=0}^{ d_2-1}a_{i\,m}x^m,\text{ where } f_0,f_i\in R[x]$ for $i=1,\ldots,k$. So, for $r_0,\ldots,r_{k}\in R$, we have since $F$ is induced by $f$,
		\begin{align*}
		F(r_0+&\sum\limits_{i=1}^{k}r_i\alfa_i)=f(r_0+\sum\limits_{i=1}^{k}r_i\alfa_i)=f_0(r_0)+ \sum\limits_{i=1}^{k}(r_if_0'(r_0)+f_i(r_0))\alfa_i \text{ (by Lemma}~\ref{CHS02})\\
		&=\sum\limits_{l=0}^{d_1-1}a_{0\,l}r_0^l+\sum\limits_{i=1}^{k}(\sum\limits_{l=1}^{d_1-1}r_i(la_{0\,l}r_0^{l-1})  +\sum\limits_{m=0}^{d_2-1}a_{i\,m}r_0^m)
		\alpha_i 
		\\
		&=F_{0}(r_0,\ldots,r_k)+\sum\limits_{i=1}^{k}F_i(r_0,\ldots,r_k)\alfa_i \text{ (by the definition of }F)
		\\
		&=b_{0}(r_0)+\sum\limits_{i=1}^{k}b_i(r_0,r_i)\alfa_i.
		\end{align*}
		Therefore,
		\begin{align*}
		\sum\limits_{l=0}^{d_1-1}a_{0\,l}r_0^ l& =  b_{0}(r_0)
		\\
		\sum\limits_{l=1}^{d_1-1}(la_{0\,l}r_0^{l-1}) r_i +\sum\limits_{m=0}^{d_2-1}a_{i\,m}r_0^m & =  b_i(r_0,r_i) \text{ for }  i=1,\ldots,k.  
		\end{align*}
		
		Hence,  since each $r_j$ varies through  all the elements  of $R$, the system of linear equations~(\ref{CHSSys}) has a solution 
		$y_{0\,l}=a_{0\,l},\,y_{i\,m}=a_{i\,m}$, %\ 0\le a_{0\,l},a_{i\,m}< m$,
		for $l=0,\ldots,d_1 -1$;  $m=0,\ldots,d_2 -1$; $i=1,\dots,k$.
		
		Finally, we can prove $(4)\Rightarrow(3)$ by reversing the previous steps. 
	\end{proof}
	\comment{
		\begin{remark}
			Keep the notation of Theorem~\ref{CHSsystemlinear}.
			\begin{enumerate}
				\item For an element $r_0+\sum\limits_{i=1}^{k}r_{i}\alpha_i$, there are  exactly $k+1$ equations in the 
				system of linear equations~(\ref{CHSSys}) corresponding to this element.
				\item In view of the necessary conditions of Corollary~\ref{CHSneccondit}, we expect to get repetitions of equations in the system~(\ref{CHSSys}). For example, let us consider the elements $r_0+\sum\limits_{i=1}^{k}r_{i}\alpha_i$ and $r_0+\sum\limits_{i=1}^{k}c_{i}\alpha_i$ of $\Ralfak$ with $c_i \ne r_i$ for some $i\ge 1$. Then,  for a polynomial function $F$, it is necessary that      $b_{0}(r_0,r_1,\ldots,r_k)= b_{0}(r_0,c_1,\ldots,c_k)$. Because, otherwise,   the system~(\ref{CHSSys}) will contain the following equations 
				\begin{align} 
				\sum\limits_{l=0}^{d_1-1}y_{0\,l}r_{0}^l& =  b_0(r_0,r_1,\ldots,r_k)\nonumber
				\\
				\sum\limits_{l=0}^{d_1-1}y_{0\,l}r_{0}^l& =  b_0(r_0,c_1,\ldots,c_k)\nonumber, 
				\end{align}
				which implies that the system~(\ref{CHSSys}) has no solution.
			\end{enumerate}
	\end{remark}} %end of comment
	\noindent {\bf Acknowledgment.} This work was supported by  
	the Austrian Science Fund FWF: P 27816-N26 and P 30934-N35. The author would like to thank  Kwok Chi Chim and Paolo Leontti
	for  valuable suggestions and comments on earlier versions of the manuscript. 
	\bibliographystyle{plain}
	\bibliography{PolyAlgDis}

\begin{thebibliography}{10}

\bibitem{Haki}
Hasan Al-Ezeh, Amr~Ali Al-Maktry, and Sophie Frisch.
\newblock Polynomial functions on rings of dual numbers over residue class
  rings of the integers.
\newblock {\em Mathematica Slovaca}, 71(5):1063--1088, 2021.

\bibitem{Cobj1}
Neal Brand.
\newblock Isomorphisms of cyclic combinatorial objects.
\newblock {\em Discrete Math.}, 78(1-2):73--81, 1989.

\bibitem{Cobj2}
Neal Brand.
\newblock Polynomial isomorphisms of combinatorial objects.
\newblock {\em Graphs Combin.}, 7(1):7--14, 1991.

\bibitem{gal}
Joel~V. Brawley and Gary~L. Mullen.
\newblock Functions and polynomials over {G}alois rings.
\newblock {\em J. Number Theory}, 41(2):156--166, 1992.

\bibitem{moti}
Z.~Chen.
\newblock On polynomial functions from {$Z_n$} to {$Z_m$}.
\newblock {\em Discrete Math.}, 137(1-3):137--145, 1995.

\bibitem{suit}
Sophie Frisch.
\newblock Polynomial functions on finite commutative rings.
\newblock In {\em Advances in Commutative Ring Theory ({F}ez, 1997)}, volume
  205 of {\em Lecture Notes in Pure and Appl. Math.}, pages 323--336. Dekker,
  New York, 1999.

\bibitem{per2}
Sophie Frisch and Daniel Krenn.
\newblock Sylow {$p$}-groups of polynomial permutations on the integers {${\rm
  mod}\,p^n$}.
\newblock {\em J. Number Theory}, 133(12):4188--4199, 2013.

\bibitem{pol1}
Gordon Keller and F.~R. Olson.
\newblock Counting polynomial functions {$({\rm mod}\ p^{n})$}.
\newblock {\em Duke Math. J.}, 35:835--838, 1968.

\bibitem{Residue}
Aubrey~J. Kempner.
\newblock Polynomials and their residue systems.
\newblock {\em Trans. Amer. Math. Soc.}, 22(2):240--266, 267--288, 1921.

\bibitem{semi}
Gerhard Kowol and Heinz Mitsch.
\newblock Polynomial functions over commutative semi-groups.
\newblock {\em Semigroup Forum}, 12(2):109--118, 1976.

\bibitem{pol3}
Gary~L. Mullen and Harlan Stevens.
\newblock Polynomial functions {$({\rm mod}\,m)$}.
\newblock {\em Acta Math. Hungar.}, 44(3-4):237--241, 1984.

\bibitem{Necha}
Alexander~A. Nechaev.
\newblock Polynomial transformations of finite commutative local rings of
  principal ideals.
\newblock 27:425--432, 1980.
\newblock transl. from 27 (1980) 885-897, 989.

\bibitem{per1}
Wilfried N\"obauer.
\newblock Gruppen von {R}estpolynomidealrestklassen nach {P}rimzahlpotenzen.
\newblock {\em Monatsh. Math.}, 59:194--202, 1955.

\bibitem{pol2}
David Singmaster.
\newblock On polynomial functions {$({\rm mod}$} {$m)$}.
\newblock {\em J. Number Theory}, 6:345--352, 1974.

\bibitem{Com1}
Jing Sun and Oscar~Y. Takeshita.
\newblock Interleavers for turbo codes using permutation polynomials over
  integer rings.
\newblock {\em IEEE Trans. Inform. Theory}, 51(1):101--119, 2005.

\bibitem{Com2}
Oscar~Y. Takeshita.
\newblock Permutation polynomial interleavers: an algebraic-geometric
  perspective.
\newblock {\em IEEE Trans. Inform. Theory}, 53(6):2116--2132, 2007.

\bibitem{mon}
Robert~F. Tichy.
\newblock Polynomial functions over monoids.
\newblock {\em Semigroup Forum}, 18(4):371--380, 1979.

\end{thebibliography}
\end{document}